\documentclass[11pt]{article}
\usepackage{authblk}
\usepackage{amssymb}
\usepackage[all]{xy}
\usepackage{mathrsfs}
\usepackage{amsmath,amssymb,amsthm}
\usepackage{extarrows}
\newtheorem{theorem}{Theorem}[]
\newtheorem{lemma}[theorem]{Lemma}

\newtheorem{corollary}[theorem]{Corollary}
\newtheorem{example}[theorem]{Example}

\def\to{\rightarrow}
\def\f{\mathfrak}

\def\c{\mathcal}
\def\b{\mathbf}
\def\r{\mathrm}
\def\bb{\mathbb}
\def\s{\mathscr}
\date{}
\begin{document}
\title{Abstract Hardy-Littlewood Maximal Inequality}
\author{Maysam Maysami Sadr\thanks{sadr@iasbs.ac.ir}\hspace{3mm}\&\hspace{2mm}Monireh Barzegar Ganji\thanks{barzegar.ganji@iasbs.ac.ir}}
\affil{Department of Mathematics, Institute for Advanced Studies in Basic Sciences, Zanjan, Iran}
\maketitle
\begin{abstract}
In this note besides two abstract versions of the Vitali Covering Lemma an abstract Hardy-Littlewood Maximal Inequality,
generalizing weak type $(1,1)$ maximal function inequality, associated to any outer measure and a family of subsets on a set is introduced.
The inequality is (effectively) satisfied if and only if a special numerical constant called Hardy-Littelwood maximal constant is finite.
Two general sufficient conditions for the finiteness of this constant are given and upper bounds for this constant
associated to the family of (centered) balls in homogeneous spaces, family of dyadic cubes in Euclidean spaces,
family of admissible trapezoids in homogeneous trees, and family of Calder\'{o}n-Zygmund sets in $(ax+b)$-group, are considered.
Also a very simple application to find some nontrivial estimates about mass density in Classical Mechanics is given.


\end{abstract}
\section{The Main Concept}\label{2005142105}
Let $X$ be a nonempty set. Denote by $\c{P}(X)$ the power set of $X$. Consider following data:
\begin{enumerate}
\item[(HL1)] For every $x\in X$ a nonempty set $\c{Q}_x\subseteq\c{P}(X)$ of subsets of $X$ such that
$x\in Q$ for every $Q\in\c{Q}_x$. We let $\f{Q}:=\{\c{Q}_x\}_x$ and $\c{Q}:=\cup_x\c{Q}_x$.
\item[(HL2)] An outer measure ${\mu}$ on $X$ such that $0<{\mu}(Q)<\infty$ for every $Q$ in $\c{Q}$.
\end{enumerate}
Then to any arbitrary set function $F:\c{Q}\to[0,\infty]$ we may associate a \emph{maximal function}
$$\b{M}^{\f{Q}}_{\mu}F:X\to[0,\infty],\hspace{10mm}\b{M}^{\f{Q}}_{\mu}F(x):=\sup_{Q\in\c{Q}_x}\frac{F(Q)}{{\mu}(Q)},$$
and a \emph{norm} $\|F\|_\c{Q}$ given by
$$\|F\|_\c{Q}:=\sup\Big\{\sum_{i=1}^nF(Q_i):Q_1,\ldots,Q_n\in\c{Q}\hspace{1mm}\text{and}\hspace{1mm}
Q_i\cap Q_j=\emptyset\hspace{1mm}\text{for}\hspace{1mm}i\neq j\Big\}.$$
With the above notations by an \emph{Abstract Hardy-Littlewood Maximal Inequality} we mean an inequality of the form
\begin{equation}\label{2005072329}
{\mu}\Big\{x\in X:\b{M}^{\f{Q}}_{\mu}F(x)>r\Big\}\leq cr^{-1}\|F\|_\c{Q}\hspace{10mm}(r>0)
\end{equation}
satisfying for every set function $F:\c{Q}\to[0,\infty]$ where $c>0$ is a constant independent from $F$.
We call the smallest constant $c$ satisfying (\ref{2005072329}) \emph{Hardy-Littlewood maximal constant} of the pair $(\f{Q},{\mu})$
and denote it by $\s{H}^\f{Q}_\mu$. If there does not exist a finite $c$ satisfying (\ref{2005072329}) then we let $\s{H}^\f{Q}_\mu:=\infty$.
The ordinary weak type $(1,1)$ maximal function inequalities follow from (\ref{2005072329}) by $F(Q)=\int_Q|f|\r{d}\mu$ for locally
integrable functions $f:X\to\bb{C}$. (see (\ref{2005130819}) below.) Note that one of the advantages of (\ref{2005072329}) rather than
ordinary Hardy-Littlewood maximal inequalities is independence of $F$ and $\|F\|_\c{Q}$ from $\mu$.

The main aim of this short note is to show that under some general conditions on $\f{Q}$ and $\mu$ the inequality (\ref{2005072329})
is satisfied for a finite constant $c$, or equivalently, $\s{H}^\f{Q}_\mu<\infty$. To find and state such general conditions covering lemmas play
the crucial role as in the ordinary Hardy-Littlewood maximal function inequalities.
In Section \ref{2005140808} we consider two \emph{Abstract Vitali Covering Lemmas}
that have their own interest. In Section \ref{2005140809} we give two general sufficient conditions for the finiteness of $\s{H}^\f{Q}_\mu$
and find some bounds for Hardy-Littlewood maximal constant of the family of (centered) balls in homogeneous spaces,
family of dyadic cubes in Euclidean spaces, family of admissible trapezoids in homogeneous trees,
and family of Calder\'{o}n-Zygmund sets in $(ax+b)$-group. We also consider a very simple application about mass density in Classical Mechanics.
\section{Abstract Vitali Covering Lemmas}\label{2005140808}
In this section we consider some basic covering lemmas that will be applied in Section \ref{2005140809}.
The proof of the following lemma is similar to the proof of \cite[Theorem 1.2]{Heinonen1}.
\begin{lemma}\label{2005102352}
\emph{\textbf{Abstract Vitali Covering Lemma: First Form.}
Let $X$ be a nonempty set, let $\c{C}\subseteq\c{P}(X)$ be a nonempty set of subsets of $X$, and let $\Delta:\c{B}\to\c{P}(X)$ be a mapping
such that $C\subseteq\Delta(C)$ for every $C\in\c{C}$. Consider the following two conditions:
\begin{enumerate}
\item[(V1)] There exist a bounded set function $\gamma:\c{C}\to[0,\infty)$ and a finite constant $\lambda>1$
such that if $C_1,C_2\in\c{C}$, $C_1\cap C_2\neq\emptyset$ and $\gamma(C_1)\leq\lambda\gamma(C_2)$
then $C_1\subseteq\Delta(C_2)$.
\item[(V2)] There exists a bounded set function $\gamma:\c{C}\to[0,\infty)$ such that for every subfamily $\c{G}$ of $\c{C}$ the maximum
of $\{\gamma(C): C\in\c{G}\}$ exists and such that if $C_1,C_2\in\c{C}$, $C_1\cap C_2\neq\emptyset$ and $\gamma(C_1)\leq\gamma(C_2)$
then $C_1\subseteq\Delta(C_2)$.
\end{enumerate}
Suppose that either (V1) or (V2) are satisfied.
Then there exists a subfamily $\c{D}\subseteq\c{C}$ such that the members of $\c{D}$ are pairwise disjoint and
\begin{equation}\label{2005102135}
\bigcup_{C\in\c{C}}C\subseteq\bigcup_{D\in\c{D}}\Delta(D).
\end{equation}}
\end{lemma}
\begin{proof}
Suppose that (V1) is satisfied.
Let $\Omega$ denote the set of all $\c{G}\subseteq\c{C}$ such that the members of $\c{G}$ are pairwise disjoint and for every $C\in\c{C}$
if $C\cap G\neq\emptyset$ for some $G\in\c{G}$ then there exists $G'\in\c{G}$ such that $\gamma(C)\leq\lambda\gamma(G')$. For any $D\in\c{C}$
with $\gamma(D)\geq\lambda^{-1}\sup_{C\in\c{C}}\gamma(C)$ we have $\{D\}\in\Omega$. Thus $\Omega$ is not empty.
$\Omega$ is partially ordered by the inclusion and it is clear that the union of any chains of families in $\Omega$ belongs to $\Omega$ too.
So by the Zorn Lemma $\Omega$ has a maximal member $\c{D}$. Let $\c{D}'$ denote the set of those members $C$ of $\c{C}$ such that for every
$D\in\c{D}$, $C\cap D=\emptyset$. Suppose that $\c{D}'\neq\emptyset$ and $D'\in\c{D}'$ be such that
$\gamma(D')\geq\lambda^{-1}\sup_{C\in\c{D}'}\gamma(C)$. Then we must have $\c{D}\cup\{D'\}\in\Omega$, a contradiction with maximality of $\c{D}$.
Thus $\c{D}'=\emptyset$ and hence we can conclude from the property stated in (V1) that the maximal family $\c{D}$ satisfies (\ref{2005102135}).

The proof in the case that (V2) is satisfied is similar: It is enough let in the above proof $\lambda=1$ and replace $\sup$ with $\max$.
\end{proof}
Note that in the above discussion we can ignore all sets $C$ with $\gamma(C)=0$. Indeed it follows from (V1) and (V2) that for such sets $C$
we have $C\subseteq\Delta(D)$ for any other set $D$. We did not exclude such cases for generality.

\begin{corollary}\label{2005122347}
\emph{\textbf{Vitali Covering Lemma.} If $\c{C}$ is a family of open balls of uniformly bounded radius in a metric space $X$
then there exists a subfamily $\c{D}\subseteq\c{C}$ with pairwise disjoint members such that $\cup_{B\in\c{C}}C\subseteq\cup_{B\in\c{D}}5B$,
where $5B$ denote the ball with the same center as $B$ and with the radius $5$ times the radius of $B$.}
\end{corollary}
\begin{proof}
It follows directly from Lemma \ref{2005102352} by $\Delta(B)=(5B)$ and $\gamma(B)=\hspace{1mm}\text{radius of}\hspace{1mm}B$.
Here we may choose $\lambda=2$.
\end{proof}
Lemma \ref{2005102352} with property (V2) will be applied in Section \ref{2005140809} when the coverings of admissible trapezoids is considered.

Now we give another covering lemma similar to Lemma \ref{2005102352} but this time the \emph{gauge set function} $\gamma$ is given and fixed
and but the exact \emph{shape} of the \emph{dilated set} $\Delta(C)$ of $C$ is not important.
\begin{lemma}\label{2004221705}
\emph{\textbf{Abstract Vitali Covering Lemma: Second Form.} Let $X$ be a set.
\begin{enumerate}
\item[(V3)] Let $\c{A}\subseteq\c{P}(X)$ be a
union closed family of subsets of $X$ and let $\gamma:\c{A}\to[0,\infty]$ be a monotone set-function on $\c{A}$ i.e. if $A,A'\in\c{A}$ and
$A\subseteq A'$ then $\gamma(A)\leq\gamma(A')$.
\item[(V4)] Let $\c{B}\subseteq\c{A}$ be a subfamily with the property that for every sequence $(B_n)_n$
of pairwise disjoint members of $\c{B}$ with $\inf_n\gamma(B_n)>0$ we have $\gamma(\cup_{n=1}^\infty B_n)=\infty$.
\end{enumerate}
Suppose that $\c{C}\subseteq\c{B}$ be a subfamily of $\c{B}$ with $\sup_{C\in\c{C}}\gamma(C)<\infty$ such that $\gamma(C)\neq0$ for every $C\in\c{C}$.
Then for every $\lambda>1$ there is a (finite or infinite) countable sequence $(C^{\lambda}_n)_n$ of pairwise disjoint members of $\c{C}$
such that if $\gamma(\cup_{C\in\c{C}}C)<\infty$ then
\begin{equation}\label{2005142100}
\cup_{C\in\c{C}}C\subseteq\cup_n\Delta^\lambda_{\gamma,\c{C}}(C^\lambda_n)
\end{equation}
and if $\gamma(\cup_{C\in\c{C}}C)=\infty$ then either $\inf_{n}\gamma(C^\lambda_n)>0$ or the inclusion (\ref{2005142100}) is satisfied.
Here $\Delta^\lambda_{\gamma,\c{C}}:\c{C}\to\c{P}(X)$ is the \emph{dilating map} defined by
\begin{equation}\label{2005142102}
\Delta^\lambda_{\gamma,\c{C}}(C):=\bigcup_{G\in\c{C},G\cap C\neq\emptyset,\gamma(G)\leq\lambda\gamma(C)}G.
\end{equation}
Moreover suppose that $\c{C}$ has the following additional property: \emph{For every subfamily $\c{G}\subseteq\c{C}$
the maximum of the set $\{\gamma(G):G\in\c{G}\}$ of real numbers exists.} Then there exists a (finite or infinite) countable sequence $(C^{1}_n)_n$ of
pairwise disjoint members of $\c{C}$ such that if $\gamma(\cup_{C\in\c{C}}C)<\infty$ then
\begin{equation}\label{2005142101}
\bigcup_{C\in\c{C}}C\subseteq\bigcup_n\Delta^1_{\gamma,\c{C}}(C^{1}_n)
\end{equation}
and if $\gamma(\cup_{C\in\c{C}}C)=\infty$ then either $\inf_{n}\gamma(C^{1}_n)>0$ or the inclusion (\ref{2005142101}) is satisfied.}
\end{lemma}
\begin{proof}
Let $\lambda>1$. By induction we choose the desired sequence $(C^{\lambda}_n)_n$. Let
$$r_0:=\sup\{\gamma(C):C\in\c{C}\}<\infty.$$
We choose $C^\lambda_1\in\c{C}$ such that $\gamma(C^\lambda_1)\geq\lambda^{-1}r_0$. Let
$$r_1:=\sup\big\{\gamma(C):C\in\c{C},C\cap C^\lambda_1=\emptyset\big\}.$$
If $r_1=0$ we have done our job: Indeed in this situation for any $C\in\c{C}$ we have $C\cap C^\lambda_1\neq\emptyset$. And
on the other hand $\gamma(C)\leq r_0\leq\lambda\gamma(C^\lambda_1)$. Thus by (\ref{2005142102}),
$C\subseteq\Delta^\lambda_{\gamma,\c{C}}(C^\lambda_1)$, and hence
$$\bigcup_{C\in\c{C}}C\subseteq\Delta^\lambda_{\gamma,\c{C}}(C^\lambda_1).$$
If $r_1\neq0$ we choose $C^\lambda_2\in\c{C}$ such that $C^\lambda_1\cap C^\lambda_2=\emptyset$ and $\gamma(C^\lambda_2)\geq\lambda^{-1}r_1$.
So suppose that we have chosen $C^\lambda_1,\ldots,C^\lambda_{i-1}$ for some $i\geq3$. We let
$$r_{i-1}:=\sup\big\{\gamma(C):C\in\c{C},C\cap(\cup_{k=1}^{i-1}C^\lambda_k)=\emptyset\big\}.$$
If $r_{i-1}=0$ we have $\cup_{C\in\c{C}}C\subseteq\cup_{k=1}^{i-1}\Delta^\lambda_{\gamma,\c{C}}(C_k)$. Otherwise we choose $C^\lambda_i$
in the way similar to $C^\lambda_1,\ldots,C^\lambda_{i-1}$, and proceed analogously.

Suppose that the procedure of choosing $C^\lambda_n$ never stops. Thus we find the sequence $(C^\lambda_n)_{n\geq1}$ of pairwise disjoint
members of $\c{C}$. Suppose that $\gamma(\cup_{C\in\c{C}}C)<\infty$.
From the assumptions on $\gamma$ it is concluded that $\inf_n r_n=0$.
Let $C$ be an arbitrary member of $\c{C}$. Then for some $i$ we must have $r_i<\gamma(C)$.
This implies that $C\cap(\cup_{k=1}^{i}C^\lambda_k)\neq\emptyset$. Then a similar argument as above shows that
$C\subseteq\cup_{k=1}^{i}\Delta(C^\lambda_k)$. Thus (\ref{2005142100}) is satisfied.

The proof of the last part is similar: Just let in the above proof $\lambda=1$ and replace $\sup$ with $\max$.
\end{proof}
\begin{example}
\emph{In the following situations the assumptions (V3) and (V4) are satisfied
\begin{enumerate}
\item[(i)] $\c{A}=\c{P}(X)$, $\gamma$ is an outer measure on $X$ and $\c{B}$ is the family of all Carath\'{e}odory measurable subsets of $X$.
\item[(ii)] $X$ is a metric space with the property that the union of every infinite
family of pairwise disjoint balls of a fixed radius is unbounded,
$\c{A}=\c{P}(X)$, $\gamma=\r{diam}$ associates to any subset of $X$ its diameter, and $\c{B}$ is the family of all balls in $X$.
Examples of such a metric space are: any compact metric space, any nonpositive constant sectional curvature form space
$\bb{M}_\delta$ \cite[Section III.4]{Chavel1} (i.e. standard Euclidean and hyperbolic spaces),
and any two dimensional complete Riemannian manifold of positive nonnegative sectional curvature (see \cite[Theorem A]{CrokeKarcher1}).
\item[(iii)] $X$ and $\c{A}$ as in (ii), $\gamma(A)=\sup_{x\in A,y\in F}d(x,y)$ where $F$ is an arbitrary fixed subset of $X$, and
$\c{B}$ is the family of all balls with radii bigger than a fixed positive number.
\end{enumerate}}
\end{example}
We still need another simple covering lemma:
\begin{lemma}\label{2005150803}
\emph{Let $X$ be a nonempty set. Let $\c{C}\subseteq\c{P}(X)$ be such that for every $C,C'\in\c{C}$ we have
$\big(C\cap C'\big)\in\big\{\emptyset,C,C'\big\}$, and such that for every $C\in\c{C}$ there is a maximal member
(with respect to inclusion) of $\c{C}$ containing $C$. Then there is a subfamily $\c{D}\subseteq\c{C}$ such that the members
of $\c{D}$ are pairwise disjoint and
$$\bigcup_{C\in\c{C}}C=\bigcup_{D\in\c{D}}D.$$}
\end{lemma}
\begin{proof}
$\c{D}$ may be considered to be the subfamily of all maximal members of $\c{C}$.
\end{proof}
As an application of simple Lemma \ref{2005150803} we consider the following result. Another application will be given in the proof
of Theorem \ref{2005150805}. By a left-closed right-open cube in $\bb{R}^n$ we mean a cube which is the cartesian product of the intervals of
the form $[a,b)$. For a cube $Q$ we denote by $L_Q$ the side length of $Q$. A \emph{dyadic} cube is a cube of the form $2^km+2^\ell[0,1)^{n}$
where $m\in\bb{Z}^n$ and $k,\ell\in\bb{Z}$. Note that for every two dyadic cubes $Q,Q'$ we have
$$Q\cap Q'\in\Big\{Q,Q',\emptyset\Big\}.$$
\begin{theorem}
\emph{(\cite[Theorem I.3]{Stein2}) Let $F$ be a nonempty closed subset of $\bb{R}^n$. Then there is a family $\c{D}$ of pairwise disjoint
left-closed right-open cubes such that $\bb{R}^n\setminus F=\cup_{Q\in\c{D}}Q$ and for every $Q\in\c{D}$,
\begin{equation}\label{2005250712}
\sqrt{n}L_Q\leq\r{dist}(Q,F)\leq4\sqrt{n}L_Q.
\end{equation}}
\end{theorem}
\begin{proof}
Let $\c{C}$ denote the family of all dyadic cube satisfying in (\ref{2005250712}). It is easily verified that the conditions of
Lemma \ref{2005150803} are satisfied and hence there is a subfamily $\c{D}\subseteq\c{C}$ such that $\cup_{Q\in\c{C}}Q=\cup_{Q\in\c{D}}Q$.
Thus now it is enough to show that $\bb{R}^n\setminus F=\cup_{Q\in\c{C}}Q$. Let $x\in\bb{R}^n\setminus F$. Let $\ell\in\bb{Z}$ be the unique
integer satisfying
$$\sqrt{n}2^\ell<\r{dist}(x,F)\leq\sqrt{n}2^{\ell+1}.$$
Let $Q$ be a dyadic cube containing $x$ with $L_Q=2^{\ell-1}$. We have
$$\r{dist}(Q,F)\leq\r{dist}(x,F)\leq4\sqrt{n}L_Q.$$
Also for every $y\in Q$ we have
$$\sqrt{n}L_Q=\sqrt{n}2^{\ell-1}=\sqrt{n}2^{\ell}-\sqrt{n}2^{\ell-1}\leq|\r{dist}(x,F)-\r{dist}(x,y)|\leq\r{dist}(y,F).$$
Thus $\sqrt{n}L_Q\leq\r{dist}(Q,F)$. Hence $Q\in\c{C}$. The proof is complete.
\end{proof}
\section{Finiteness of $\s{H}^\f{Q}_\mu$, Its Bounds, and Some Examples}\label{2005140809}
Throughout this section we suppose the notations of Section \ref{2005142105} specially (HL1) and (HL2).
We begin by a lower bound of $\s{H}^\f{Q}_\mu$:
\begin{theorem}\label{2005150804}
\emph{Suppose $\f{Q}$ has the property that $\c{Q}_x=\{Q\in\c{Q}: x\in Q\}$ for every $x\in X$. Then $\s{H}^\f{Q}_\mu\geq1$.}
\end{theorem}
\begin{proof}
Suppose $Q\in\c{Q}$ be arbitrary and fixed and let the set function $F$ on $\c{Q}$ be defined by $F(Q)=\mu(Q)$ and $F(Q')=0$ for every
$Q'\in\c{Q}$ with $Q'\neq Q$. For every $\epsilon$ with $0<\epsilon<1$
$$\{x\in X:\b{M}^\f{Q}_\mu F(x)>1-\epsilon\}=Q.$$
On the other hand we have $\|F\|_\c{Q}=F(Q)$. Thus if $c>0$ satisfies (\ref{2005072329}) we must have $c\geq1$. The proof is complete.
\end{proof}
We have the following two elementary theorems.
\begin{theorem}\label{2005220800}
\emph{Let the data $\f{Q}',\c{Q}',\c{Q}'_x$ be similar to the data $\f{Q},\c{Q},\c{Q}_x$ with the same properties as in (HL1) and (HL2).
\begin{enumerate}
\item[(i)] If $\c{Q}_x\subseteq\c{Q}'_x$ for every $x\in X$ then $\s{H}^{\f{Q}}_\mu\leq\s{H}^{\f{Q}'}_\mu$.
\item[(ii)] Let $\f{Q}\dot{\cup}\f{Q}'$ denote the family $\{\c{Q}_x\cup\c{Q}'_x\}_x$ of families of subsets of $X$. Then
$$\max\big(\s{H}^{\f{Q}}_\mu,\s{H}^{\f{Q}'}_\mu\big)\leq\s{H}^{\f{Q}\dot{\cup}\f{Q}'}_\mu\leq\s{H}^{\f{Q}}_\mu+\s{H}^{\f{Q}'}_\mu.$$
\end{enumerate}}
\end{theorem}
\begin{proof}
(i) is easily verified. It is clear that the value $\sup_{Q''\in\c{Q}_x\cup\c{Q}'_x}\frac{F(Q'')}{\mu(Q'')}$ is the maximum of
$\sup_{Q\in\c{Q}_x}\frac{F(Q)}{\mu(Q)}$ and $\sup_{Q'\in\c{Q}'_x}\frac{F(Q')}{\mu(Q')}$. Thus $\b{M}^{\f{Q}\dot{\cup}\f{Q}'}_\mu
=\max(\b{M}^{\f{Q}}_\mu,\b{M}^{\f{Q}'}_\mu)$ and hence
$$\big\{x:\b{M}^{\f{Q}\dot{\cup}\f{Q}'}_{\mu}F(x)>r\big\}=\big\{x:\b{M}^{\f{Q}}_{\mu}F(x)>r\big\}
\cup\big\{x:\b{M}^{\f{Q}'}_{\mu}F(x)>r\big\}.$$
Thus if $c>0$ satisfies (\ref{2005072329}) and $c'>0$ satisfies the analogue of (\ref{2005072329}) for $\f{Q}'$ then any constant $b$ with
$\max(c,c')\leq b\leq(c+c')$ must satisfies the analogue of (\ref{2005072329}) for $\f{Q}\dot{\cup}\f{Q}'$.
\end{proof}
\begin{theorem}
\emph{Let $\mu'$ be another outer measure on $X$ satisfying $0<\mu'(Q)<\infty$ for $Q\in\c{Q}$.
\begin{enumerate}
\item[(i)] If $\mu\leq\mu'$ then $\s{H}^{\f{Q}}_{\mu'}\leq\s{H}^{\f{Q}}_\mu$.
\item[(ii)] $\s{H}^{\f{Q}}_{\mu+\mu'}\leq\min\big(\s{H}^{\f{Q}}_\mu,\s{H}^{\f{Q}}_{\mu'}\big)$.
\item[(iii)] $\s{H}^{\f{Q}}_{s\mu}=s^{-1}\s{H}^{\f{Q}}_\mu$ for any real number $s>0$.
\end{enumerate}}
\end{theorem}
\begin{proof}
(i) and (iii) are easily verified. (ii) follows from (i) immediately.
\end{proof}
The following theorem gives an upper bound for $\s{H}^{\f{Q}}_{\mu}$ under some mild conditions.
\begin{theorem}\label{2005122144}
\emph{Besides (HL1) and (HL2) suppose that the following condition is satisfied.
\begin{enumerate}
\item[(HL3)] If for a sequence $(Q_n)_n$ of pairwise disjoint members of $\c{Q}$ we have $\inf_{n}{\mu}(Q_n)>0$ then
${\mu}(\cup_n Q_n)=\infty$. (Note that this condition is automatically satisfied if the members of $\c{Q}$ are
Carath\'{e}odory measurable with respect to $\mu$.)
\end{enumerate}
Suppose that the constant $c>0$ satisfies the following condition.
\begin{enumerate}
\item[(HL4)] There exists a constant $\lambda>1$ such that ${\mu}(\tilde{Q})\leq c{\mu}(Q)$ for every $Q\in\c{Q}$ where
$$\tilde{Q}:=\bigcup_{R\in\c{Q},R\cap Q\neq\emptyset,\mu(R)\leq\lambda\mu(Q)}{R}.$$
\end{enumerate}
Then we have $\s{H}^{\f{Q}}_{\mu}\leq c$.}
\end{theorem}
\begin{proof}
Suppose that $F:\c{Q}\to[0,\infty]$ is a set function with $\|F\|_\c{Q}<\infty$. We must show that the inequality (\ref{2005072329}) is satisfied.
Let $r>0$ be arbitrary and fixed and let
$$A:=\{x\in X:\b{M}^{\f{Q}}_{\mu}F(x)>r\}.$$
For every $x\in A$ there exists $Q_x\in\c{Q}_x$ such that $r{\mu}(Q_x)\leq F(Q_x)\leq\|F\|_\c{Q}$. Thus we have
\begin{equation}\label{2005150801}
A\subseteq\bigcup_{x\in A}Q_x\hspace{10mm}\sup_{x\in A}\mu(Q_x)\leq r^{-1}\|F\|_\c{Q}<\infty
\end{equation}
Also if $(x_n)_n$ is a sequence in $A$ such that $Q_{x_i}\cap Q_{x_j}=\emptyset$ for $i\neq j$ then from the inequality
$$\sum_{i=1}^n\mu(Q_{x_i})\leq r^{-1}\sum_{i=1}^nF(Q_{x_i})\leq r^{-1}\|F\|_\c{Q}$$
we have that $\inf_{n}\mu(Q_{x_n})=0$. Thus it is concluded from Lemma \ref{2004221705}
by $\c{B}=\c{P}(X)$, $\gamma=\mu$, and $\c{C}=\{Q_x\}_{x\in A}$ that there is a (finite or infinite) countable sequence
$(x_n)_n$ in $A$ such that $Q_{x_i}\cap Q_{x_j}=\emptyset$ for $i\neq j$ and such that $\cup_{x\in A}Q_x\subseteq\cup_{n}\tilde{Q}_{x_n}$. We have
$$\mu(A)\leq\mu(\cup_{n}\tilde{Q}_{x_n})\leq\sum_n\mu(\tilde{Q}_{x_n})\leq c\sum_n\mu(Q_{x_n})\leq cr^{-1}\sum_nF(Q_{x_n})\leq cr^{-1}\|F\|_\c{Q}.$$
The proof is complete.
\end{proof}
Note that the condition (HL4) is only explicitly depend on $\c{Q}$ not on $\f{Q}$. Thus if $c$ satisfies (HL4) and we replace $\c{Q}_x$ with another
family $\c{Q}'_x$ of sets containing $x$ and just we have $\cup_{x\in X}\c{Q}'_x=\c{Q}$ then $c$ is still satisfies (HL4).

Recall that a quasimetric space is a set $X$ together with a quasimetric $d:X\times X\to[0,\infty)$ satisfying
$d(x,y)=0$ iff $x=y$, $d(x,y)=d(y,x)$ and
\begin{equation}\label{2005122216}
d(x,z)\leq K\big(d(x,y)+d(y,z)\big)
\end{equation}
for some finite constant $K\geq1$. A quasimetric space has a canonical topology whose its base is given by the set of all open balls with respect
to the quasimetric $d$. A quasimetric space $X$ endowed with a Borel measure $\mu$ is called of homogeneous type if there exist finite constants
$\alpha,\beta>1$ such that for every every open ball $B$
of $X$ we have $0<\mu(B)<\infty$ and
\begin{equation}\label{2005122217}
\mu(\alpha B)\leq\beta\mu(B)
\end{equation}
Euclidean spaces with Euclidean distance and Lebesgue measure, compact Riemannian manifolds, Heisenberg group as a Riemannian manifold
with left- or right-invariant metric are some examples of spaces of homogeneous type. For more examples see
\cite{CoifmanWeiss1,FollandStein1,SteinMurphy1,Wallach1}. In the following result we prove a type of well-known maximal function inequalities
(see the mentioned references) for homogeneous spaces as a corollary of Theorem \ref{2005122144}.
(For a Borel measure $\mu$ on a topological space $X$ we denote by the same symbol $\mu$ the canonical outer measure induced by $\mu$ on $X$.)
\begin{theorem}\label{2005150800}
\emph{Let $(X,d,\mu)$ be a space of homogeneous type. Let $\r{ball}$ denote the family $\f{Q}$ where for every $x\in X$, $\c{Q}_x$
be the set of all open balls of $X$ with center $x$ (respectively, the set of all balls of $X$ containing $x$).
Then we have $\s{H}^\r{ball}_\mu\leq\beta^m$ where $m$ is the smallest natural number with $\alpha^m\geq2K(4K^2+1)$ and where
$K,\alpha,\beta$ are as in (\ref{2005122216}) and (\ref{2005122217}). In particular for every function $f:X\to\bb{C}$ in $\b{L}^1(\mu)$
if $\b{M}^\r{ball}_\mu f$ denotes the \emph{centered-ball maximal function} (respectively, \emph{ball maximal function})
of $f$ with respect to $\mu$ given by
$$\b{M}^\r{ball}_{\mu}f(x):=\sup_{Q\in\c{Q}_x}\frac{1}{\mu(Q)}\int_{Q}|f|\r{d}\mu\hspace{10mm}(x\in X)$$
then for every $r>0$ we have
\begin{equation}\label{2005130819}
\mu\{x:\b{M}^\r{ball}_{\mu}f(x)>r\}\leq r^{-1}\beta^m\|f\|_{\b{L}^1(\mu)}.
\end{equation}}
\end{theorem}
\begin{proof}
Let $Q\in\c{Q}$ and let $\tilde{Q}$ be as in (HL4) for $\lambda=\beta$. Let
$$\c{E}_Q:=\{R\in\c{Q}: R\cap Q\neq\emptyset,\mu(R)\leq\beta\mu(Q)\big\},$$
and let $s:=\sup_{R\in\c{E}_Q}\r{r}_R$. Since $Q\in\c{E}_Q$, $\r{r}_Q\leq s$. If $s=\infty$ there is a sequence $(R_n)_n$
in $\c{E}_Q$ such that $R_n\subseteq R_{n+1}$ and $\r{r}_{R_n}\to\infty$. Then it is easily concluded that $\tilde{Q}=X$ and
$\mu(\tilde{Q})=\lim_{n\to\infty}\mu(R_n)$ and hence $\mu(\tilde{Q})\leq\beta\mu(Q)$. Suppose that $s\neq\infty$. For every $R\in\c{E}_Q$
$$d(\r{c}_Q,\r{c}_R)\leq K(\r{r}_Q+\r{r}_R)\leq2Ks.$$
Let $Q'\in\c{E}_Q$ be such that $\r{r}_{Q'}\geq2^{-1}s$. For every $R\in\c{E}_Q$, $d(\r{c}_{R},\r{c}_{Q'})\leq4K^2s$. Hence for every
$x\in\tilde{Q}$ we have $d(x,\r{c}_{Q'})\leq Ks(4K^2+1)\leq2\r{r}_{Q'}K(4K^2+1)$. This means that
$$\tilde{Q}\subseteq 2K(4K^2+1)Q'\subseteq\alpha^mQ'$$
where $m$ is as in the statement of the theorem. Now by (\ref{2005122217}) we have $\mu(\tilde{Q})\leq\beta^m\mu(Q)$. Thus the condition
of (HL4) is satisfied for $\lambda=\beta$ and $c=\beta^m$. Also (\ref{2005130819}) follows from (\ref{2005072329})
by $F(Q):=\int_{Q}|f|\r{d}\mu$. Note that we have $\|F\|_\c{Q}\leq\|f\|_{\b{L}^1(\mu)}$.
\end{proof}
By the notations of Theorem \ref{2005150800} it follows immediately from the theorem that
if $X=\bb{R}^n$ be the $n$-dimensional Euclidean space with standard
Euclidean metric and $\mu=\b{m}^n$ be the $n$-dimensional Lebesgue measure then $\s{H}^\r{ball}_{\b{m}^n}\leq2^{4n}$. (It has been computed by
$K=1,\alpha=2,\beta=2^n$.) If we consider $X=\bb{R}^n$ with the maximum distance $d$
(that is $d\big((x_i),(y_i)\big)=\max_{i=1}^n|x_i-y_i|$) then the balls with respect to $d$ are usual cubes in $\b{R}^n$
and also again we have $\s{H}^\r{cube}_{\b{m}^n}\leq2^{4n}$.
\begin{theorem}\label{2005150805}
\emph{Together with the assumptions (HL1) and (HL2) suppose that $\c{Q}$ satisfies the following five \emph{dyadical} conditions.
\begin{enumerate}
\item[(D1)] The members of $\c{Q}$ are Carath\'{e}odory measurable with respect to $\mu$.
\item[(D2)] $\c{Q}$ is countable.
\item[(D3)] $(Q\cap Q')\in\{Q,Q',\emptyset\}$ for every $Q,Q'\in\c{Q}$.
\item[(D4)] $\mu(Q)\lneq\mu(Q')$ for every $Q,Q'\in\c{Q}$ with $Q\subsetneq Q'$.
\item[(D5)] For every $\epsilon>0$ the set $\{\mu(Q):Q\in\c{Q},\mu(Q)>\epsilon\}$ has no limit point in $(0,\infty)$.
\end{enumerate}
Then we have $\s{H}^\f{Q}_\mu\leq1$.}
\end{theorem}
\begin{proof}
Let $F,A,Q_x$ be as in the proof of Theorem \ref{2005122144}. So (\ref{2005150801}) is satisfied. It follows from (D4) and (D5) that
every member of $\c{C}:=\{Q_x\}_{x\in A}$ is contained in a maximal member of $\c{C}$. Thus by Lemma \ref{2005150803} and (D2)
there is a finite or countable sequence $(x_n)_n$ in $A$ such that $Q_n\cap Q_m=\emptyset$ for $n\neq m$
and such that $\cup_{x\in A}Q_x=\cup_{n}Q_{x_n}$. We have
$$\mu(A)\leq\mu(\cup_{x\in A}Q_x)=\mu(\cup_{n}Q_{x_n})=\sum_n\mu(Q_{x_n})\leq\sum_nr^{-1}F(Q_{x_n})\leq r^{-1}\|F\|_\c{Q}.$$
The proof is complete.
\end{proof}
\begin{corollary}
\emph{Let $\r{dydc}$ denote any family $\f{Q}$ where $\c{Q}$ is a family of dyadic cubes on Euclidean space $X=\bb{R}^n$
described for instance in \cite{LernerNazarov1}. Then we have
$$\s{H}^\r{dydc}_{\b{m}^n}=1.$$}
\end{corollary}
\begin{proof}
It follows directly from Theorems \ref{2005150804} and \ref{2005150805}.
\end{proof}
Now we find an upper bound for Hardy-Littlewood maximal constant of the family of admissible
trapezoid in (degree-) homogeneous trees. We sketchy recall some of the definitions and basic properties.
Our main references here are \cite{ArdittiTabaccoVallarino1,HebischSteger1}.
Let $\bb{T}$ be a \emph{homogeneous tree} of degree $q+1$ ($q\geq1$). Thus $\bb{T}$ is just a connected graph without loop
and such that any vertex of $\bb{T}$ has degree $q+1$. The canonical metric $d$ on $\bb{T}$ associates to every two vertexes the
length of the smallest path between them. Fix a (both-side) infinite geodesic $\f{g}$ in $\bb{T}$
(that is just a subgraph of $\bb{T}$ which itself is a homogeneous tree of degree $2$) and a (surjective one-to-one) mapping
$\r{N}:\f{g}\to\bb{Z}$ such that $d(x,y)=|\r{N}(x)-\r{N}(y)|$ for every $x,y\in\f{g}$. Then the \emph{level} $\ell:\bb{T}\to\bb{Z}$
is defined by $\ell(x):=\r{N}(x')-d(x',x)$ where $x'$ is the closest point of $\f{G}$ to $x$. The canonical weighted counting measure $\tau$
on $\bb{T}$ is defined by $\tau\{x\}:=q^{\ell(x)}$ for every $x\in\bb{T}$. To ignore the trivial case we suppose that $q\geq2$.
An \emph{admissible trapezoid} is a subset $T\subseteq\bb{T}$ such that there are $x_T\in\bb{T}$ and $h_T\geq1$ with
$$T=\big\{x\in\bb{T}:\ell(x)=\ell(x_T)-d(x_T,x),h_T\leq\ell(x_T)-\ell(x)<2h_T-1\big\}.$$
$h_T$ is called the \emph{height} of $T$ and the $\omega(T):=q^{\ell(x_T)}$ is called the \emph{width} of $T$.
Also any singleton $T=\{x_T\}$ is called an admissible trapezoid with height $h_T:=1$. It is not hard to see that for any admissible
trapezoid $T$ we have
\begin{equation}\label{2005160804}
\tau(T)=h_T\omega(T)
\end{equation}
For any admissible trapezoid $T$ the subset
$$\tilde{T}:=\big\{x\in\bb{T}:\ell(x)=\ell(x_T)-d(x_T,x),2^{-1}h_T\leq\ell(x_T)-\ell(x)<4h_T\big\}$$
is called the \emph{envelope} of $T$. It is proved that \cite[Proposition 2]{ArdittiTabaccoVallarino1}
\begin{equation}\label{2005160805}
\tau(\tilde{T})\leq4\tau(T).
\end{equation}
Also we know that \cite[Proposition 3]{ArdittiTabaccoVallarino1} for admissible trapezoids $T_1,T_2$ if $T_1\cap T_2\neq\emptyset$
and if $\omega(T_1)\leq\omega(T_2)$ then $T_1\subseteq\tilde{T}_2$. Let $\c{T}$ denote the family of all admissible trapezoids in $\bb{T}$
and let $\r{trpzd}$ denote the family $\{\c{T}_x\}_{x\in\bb{T}}$ where $\c{T}_x$ is the family of all trapezoids containing $x$
(resp. the family of all trapezoids $T$ with $x_T=x$).
\begin{theorem}\label{2005220801}
\emph{For every homogeneous tree we have $\s{H}^\r{trpzd}_\tau\leq4$.}
\end{theorem}
\begin{proof}
For any $F:\c{T}\to[0,\infty]$ with $\|F\|_\c{T}<\infty$ let $A:=\{x:\b{M}^\r{trpzd}_\tau F(x)>r\}$. Then as in the proof of
Theorem \ref{2005122144} there exists a subfamily $\c{C}\subseteq\c{T}$ such that
$$A\subseteq\bigcup_{T\in\c{C}}T\hspace{10mm}\tau(T)\leq r^{-1}F(T)\leq r^{-1}\|F\|_\c{T}\hspace{5mm}(T\in\c{C})$$
Thus by (\ref{2005160804}) we have $\sup_{T\in\c{C}}\omega(T)<\infty$. It follows from this boundedness, the discreteness of $\omega$,
and the above mentioned fact, that the assumptions of (V2) in Lemma \ref{2005102352} are satisfied where $\gamma=\omega$ and $\Delta(T)=\tilde{T}$.
Thus it follows from the lemma and the countability of $\c{T}$ that there is a finite or countable pairwise disjoint sequence $(T_n)_n$
in $\c{C}$ such that $\cup_{T\in\c{C}}T\subseteq\cup_n\tilde{T}_n$. Now by (\ref{2005160805}) we have
$$\tau(A)\leq\tau(\cup_{T\in\c{C}}T)\leq\tau(\cup_n\tilde{T}_n)\leq\sum_n\tau(\tilde{T}_n)\leq4\sum_n\tau(T_n)
\leq4r^{-1}\sum_nF(T_n)\leq4r^{-1}\|F\|_\c{T}$$
The proof is complete.
\end{proof}
The inequality (\ref{2005072329}) following Theorem \ref{2005220801} recovers the inequalities in \cite[Section 3]{ArdittiTabaccoVallarino1}.

As another computational example we now find an upper bound for Hardy-Littlewood maximal constant of the family of
Calder\'{o}n-Zygmund sets in $(ax+b)$-group.
Our main reference here is \cite{LiuVallarinoYang1}. Recall that the so-called $(ax+b)$-group $\bb{A}_n$ is a Lie group which
has as underlying space the Euclidean space $\bb{R}^{n+1}$ for $n\geq1$. The group operation on $\bb{A}_n$ is given by
$(a,t).(a', t')=(a+ e^ta', t+t')$ for $a,a'\in\bb{R}^n,t,t'\in\bb{R}$. The canonical left-invariant Riemannian metric on $\bb{A}_n$ is given by
$\r{d}s^2\equiv e^{-2t}(\r{d}a^2)+\r{d}t^2$. The right Haar measure on $\bb{A}_n$ coincides with the $(n+1)$-dimensional
Lebesgue measure $\b{m}^{n+1}$. Then $\bb{A}_n$ together with minimum geodesic distance induced by the mentioned metric and the measure
measure $\b{m}^{n+1}$ is the prototype of \emph{nonhomogeneous} spaces with \emph{exponential growth}. A \emph{Calder\'{o}n-Zygmund set}
in $\bb{A}_n$ is a subset $Z$ of the form $Q\times[t-r, t+r)$ where $Q$ is a (left-closed right-open) cube in $\bb{R}^n$ with the length side $L$
such that
\begin{equation}\label{2005201542}
e^2 e^tr\leq L<e^8 e^tr\quad\text{if}\quad r<1
\end{equation}
\begin{equation}\label{2005201543}
e^te^{2r}\leq L<e^t e^{8r}\quad\text{if}\quad r\geq1
\end{equation}
Note that $\b{m}^{n+1}(Z)=L^{n}2r$. We denote by $\c{Z}$ the set of all Calder\'{o}n-Zygmund setts in $\bb{A}_n$.
\begin{lemma}\label{2005160806}
\emph{Let $Z\in\c{Z}$. Then there exists a rectangle $Z^*\subset\bb{A}_n$ such that
$$\bigcup_{Z'\in\c{Z},Z'\cap Z\neq\emptyset,\b{m}^{n+1}(Z')\leq2\b{m}^{n+1}(Z)}Z'\quad\subseteq Z^*,
\quad\b{m}^{n+1}(Z^*)\leq4^{n+2}e^{24n}\b{m}^{n+1}(Z).$$}
\end{lemma}
\begin{proof}
Let $Z'=Q'\times[t'-r',t'+r')$ be in $\c{Z}$ such that $Z\cap Z'\neq\emptyset$ and $\b{m}^{n+1}(Z')\leq2\b{m}^{n+1}(Z)$.
Also let $L'$ denote the side length of $Q'$.  Then
\begin{equation}\label{2005201544}
{L'}^n2r'\leq2L^n2r
\end{equation}
By $[t'-r',t'+r')\cap[t-r,t+r)\neq\emptyset$ we have
\begin{equation}\label{2005201545}
t-r-r'\leq t'.
\end{equation}
We have the following four possible situations:

(i) $r<1$ and $r'<1$. It follows from (\ref{2005201542}) and (\ref{2005201544}) that
$$e^{2n}e^{t'n}{r'}^n2r'\leq2e^{8n}e^{tn}{r}^n2r.$$
By (\ref{2005201545}) we have $t-2\leq t'$. Thus by the above inequality $r'^{(n+1)}\leq2e^{8n}r^{(n+1)}$ and hence
$r'\leq\lambda_1 r$ where $\lambda_1:=2^{\frac{1}{n+1}}e^{\frac{8n}{n+1}}$. By applying (\ref{2005201542}) two times
\begin{equation*}
L'\leq e^8e^{t'}r'\leq e^8e^te^2\lambda_1r\leq e^8\lambda_1L.
\end{equation*}
Thus $L'\leq\eta_1 L$ where $\eta_1:=e^8\lambda_1$.

(ii) $r<1$ and $r'\geq1$. By the methods as in (i) we find that $r'\leq\lambda_2 r$ where $\lambda_2:=2e^{8n}$. Also it follows directly from
(\ref{2005201544}) that $L'\leq\eta_2L$ where $\eta_2:=2^\frac{1}{n}$.

(iii) $r\geq1$ and $r'<1$. Then we have trivially $r'\leq\lambda_3r$ where $\lambda_3:=1$. Also it follows from  (\ref{2005201542}) and
(\ref{2005201543}) that $L'\leq\eta_3L$ where $\eta_3:=e^8$.

(iv) $r\geq1$ and $r'\geq1$. We have $r'<\lambda_4 r$ where $\lambda_4:=9$. If $r\leq r'$ by (\ref{2005201544}) we have
$L'\leq2^\frac{1}{n}L$. If $r'\geq\frac{r}{9}$ again by (\ref{2005201544}) we have $L'\leq18^\frac{1}{n}L$. If $r'\leq\frac{r}{9}$ then
by applying (\ref{2005201543}) two times we have
$$L'\leq e^{t'}e^{8r'}\leq e^te^re^{r'}e^{8r'}=e^te^{9r'}e^r\leq e^te^{2r}\leq L.$$
Thus $L'\leq\eta_4 L$ where $\eta_4:=18^\frac{1}{n}$.

Therefore we have $\lambda:=\max_{i=1}^4\lambda_i=2e^{8n}$ and $\eta:=\max_{i=1}^4\eta_i=e^82^{\frac{1}{n+1}}e^{\frac{8n}{n+1}}$.
Let $Q^*$ be the cube in $\bb{R}^n$ which its center coincides with the center of $Q$ and its side length is $(1+2\eta)L$.
Then since $Q'\cap Q\neq\emptyset$ we have $Q'\subseteq Q^*$. On the other hand since $[t'-r',t'+r')\cap[t-r,t+r)\neq\emptyset$ we have
$$[t'-r',t'+r')\subset[t-(1+2\lambda)r,t+(1+2\lambda)r).$$
Thus if we let
$$Z^*:=Q^*\times[t-(1+2\lambda)r,t+(1+2\lambda)r)$$
then $Z'\subseteq Z^*$. We also have
$$\b{m}^{n+1}(Z^*)=(1+2\eta)^nL^n2(1+2\lambda)r\leq4^{n+2}e^{24n}\b{m}^{n+1}(Z).$$
\end{proof}
Let $\r{CZ}$ denote the family $\{\c{Z}_x\}_{x\in\bb{A}_n}$ where $\c{Z}_x$ denotes the family of all Calder\'{o}n-Zygmund sets containing $x$.
\begin{theorem}
\emph{For the $(n+1)$-dimensional $(ax+b)$-group $\bb{A}_n$ we have
$$\s{H}^{\r{CZ}}_{\b{m}^{n+1}}\leq4^{n+2}e^{24n}.$$}
\end{theorem}
\begin{proof}
It follows immediately from Lemma \ref{2005160806} and Theorem \ref{2005122144}.
\end{proof}
The above result recover the first part of \cite[Proposition 2.2]{LiuVallarinoYang1}. Note that in \cite{LiuVallarinoYang1} the cubes placed in the
first $n$-components of Calder\'{o}n-Zygmund sets are dyadic cubes. Thus by Theorem \ref{2005220800}(i) our result is stronger than the mentioned
result of \cite{LiuVallarinoYang1}.

At the end of this note we consider a possible application of the inequality (\ref{2005072329}) to find some estimates for
mass density in (physical) space:

Suppose we are given a distribution of mass in Euclidean space $\bb{R}^3$. Such a distribution may be distinguished by a Borel measure $F$
on $\bb{R}^3$. Due to possibility of existence of massive point particles $F$ might be a singular measure with respect to $3$-dimensional
Lebesgue measure. For instance if we have $N$ particles with masses $m_1,\ldots,m_N$ in unit $\r{kg}$
placed at $x_1,\ldots,x_N$ in $\bb{R}^3$ then $F=\sum_{i=1}^Nm_i\delta_{x_i}$. Anyway, for every such a measure $F$ distinguishing mass
distribution, the maximal function $\b{M}_{\b{m}^3}^\r{cube}F$ with respect to the family of cubes (or other reasonable families)
may be interpreted as the \emph{maximal mass density function} in unit $\r{kg}/\r{m}^3$.
Suppose that the amount of the total mass presented in the space is $F(\bb{R}^3)=M$. Then it is clear that $\|F\|_\c{Q}=M$.
Also from the inequality (\ref{2005072329}) the following is immediately deduced:
\begin{enumerate}
\item[]
Suppose that we have a box with side length $1$ meter. Then the volume of the set of all points $x$ such that the box can be placed in the space
in the way that contains $x$ and includes more than $\alpha M$ $\r{kg}$ of mass, is less than
$\alpha^{-1}\s{H}_{\b{m}^3}^\r{cube}$. Here $\alpha$ is a value in unit $\r{m}^{-3}$ with $0<\alpha<1$. Note that $\s{H}_{\b{m}^3}^\r{cube}$
is a dimension less quantity. Note also that this estimate is independent of the \emph{shape} of the distribution $F$
and even of the total mass $M$. So it is valid in every \emph{time} for any system of particles or a fluid body moving in space
and probably losing or gaining mass.
\end{enumerate}
The above analysis may be applied in many similar problems. For instance suppose that this time $F$ gives the number of particles in the space.
Thus if we have $N$ particles placed at $x_1,\ldots,x_N$ in space then $F:=\sum_{i=1}^N\delta_{x_i}$ and $\|F\|_\c{Q}=N$.
In this case $\b{M}_{\b{m}^3}^\r{cube}F$ may be interpreted as \emph{maximal number function of particles per unit volume}.
\end{document}